\documentclass[11pt]{amsart}

\usepackage[margin=2.5cm]{geometry}
\usepackage{color}
\usepackage{mathrsfs}
\usepackage{mathtools}
\usepackage{amsmath}
\usepackage{amssymb}
\usepackage{enumitem}
\usepackage{bbm}
\usepackage{esint}
\usepackage{nicefrac}
\numberwithin{equation}{section}
\usepackage[colorlinks,citecolor=green,linkcolor=red]{hyperref}

\usepackage[latin1]{inputenc}
\usepackage{tcolorbox}

\newtheorem{theorem}{Theorem}[section]

\newtheorem{lemma}[theorem]{Lemma}
\newtheorem{proposition}[theorem]{Proposition}

\theoremstyle{definition}
\newtheorem{definition}[theorem]{Definition}

\newtheorem{remark}[theorem]{Remark}

\newcommand{\N}{\mathbb{N}}

\newcommand{\R}{\mathbb{R}}

\newcommand{\sfd}{{\sf d}}

\newcommand{\rr}{\mathbb R}
\newcommand{\restr}[1]{\lower3pt\hbox{$|_{#1}$}}

\newcommand{\eps}{\varepsilon}  
\newcommand{\nchi}{{\raise.3ex\hbox{$\chi$}}}

\newcommand{\fr}{\hfill$\blacksquare$}  

\newcommand{\Leb}[1]{{\mathcal L}^{#1}}  

\newcommand{\LIP}{\mathrm{LIP}}

\newcommand{\lip}{\mathrm{lip}}
\newcommand{\esssup}{{\rm ess}\sup}
\newcommand{\essinf}{{\rm ess}\inf}
\newcommand{\diam}{\mathrm{diam}}

\newcommand{\RCD}{{\sf RCD}}
\newcommand{\CD}{{\sf CD}}

\newcommand{\mm}{\mathfrak m}

\renewcommand{\limsup}{\varlimsup}
\renewcommand{\liminf}{\varliminf}
\renewcommand{\d}{{\rm d}}
\newcommand{\X}{{\rm X}}

\newcommand{\Xdm}{(\X,\sfd,\mm)}

\newcommand{\vol}{{\rm Vol}}

\newcommand{\Per}{{\rm Per}}

\newcommand{\la}{{\big\langle}}                  
\newcommand{\ra}{{\big\rangle}}

\renewcommand{\phi}{\varphi}

\newcommand{\cS}{\mathcal{S}}
\renewcommand{\ss}{\mathbb S}

\newcommand{\mres}{\mathbin{\vrule height 1.6ex depth 0pt width 0.13ex\vrule height 0.13ex depth 0pt width 1.3ex}}

\title[Quantitative stability in distribution for the Sobolev inequality]{Quantitative stability in distribution for the Sobolev inequality under curvature dimension condition}

\author[Max Fathi]{Max Fathi} 
\address{Universit\'e Paris Cit\'e and Sorbonne Universit\'e, CNRS, Laboratoire Jacques-Louis Lions and Laboratoire de Probabilit\'es, Statistique et Mod\'elisation, F-75013 Paris, France \newline
and DMA, \'Ecole normale sup\'erieure, Universit\'e PSL, CNRS, 75005 Paris, France \newline
and Institut Universitaire de France}
\email{mfathi@lpsm.paris}

\author[Ivan Yuri Violo]{Ivan Yuri Violo}
\address{Universit\'a di Pisa, Dipartimento di Matematica, Largo Bruno Pontecorvo 5,
56127 Pisa, Italy}
\email{{ivanyuri.violo@dm.unipi.it}}
\newcommand{\Ig}{\mathcal I^{\flat}_g}

\usepackage[numbers,sort&compress]{natbib}
\let\oldthebibliography\thebibliography
\renewcommand{\thebibliography}[1]{%
  \oldthebibliography{#1}%
  \footnotesize 
}

\makeatletter
\def\@tocline#1#2#3#4#5#6#7{\relax
  \ifnum #1>\c@tocdepth 
  \else
    \par \addpenalty\@secpenalty\addvspace{#2}%
    \begingroup \hyphenpenalty\@M
    \@ifempty{#4}{%
      \@tempdima\csname r@tocindent\number#1\endcsname\relax
    }{%
      \@tempdima#4\relax
    }%
    \parindent\z@ \leftskip#3\relax \advance\leftskip\@tempdima\relax
    \rightskip\@pnumwidth plus4em \parfillskip-\@pnumwidth
    #5\leavevmode\hskip-\@tempdima
      \ifcase #1
       \or\or \hskip 1em \or \hskip 2em \else \hskip 3em \fi%
      #6\nobreak\relax
    \dotfill\hbox to\@pnumwidth{\@tocpagenum{#7}}\par
    \nobreak
    \endgroup
  \fi}
\makeatother

\begin{document}

\begin{abstract}
   The goal of this note is to investigate quantitative stability properties of the critical Sobolev inequality in   $\CD(N-1,N)$ metric measure spaces. Assuming that the optimal constant for the  inequality is almost the same as the one of the round sphere, we show that the cumulative  distribution of  any almost extremal function is  close, in Wasserstein distance, to the one of an Aubin-Talenti bubble on the round  sphere. We obtain similar results for the log Sobolev inequality and the spectral gap under various curvature and dimension assumptions. In all cases we obtain a quantitative stability with sharp exponent. 
\end{abstract}
    
	\maketitle

\tableofcontents

	\section{Introduction}

    We investigate a particular form of stability for spaces with positive Ricci curvature and almost maximal sharp Sobolev constant. The goal is to show that under certain assumptions on curvature and dimension (in the sense of Bakry and Emery), if the sharp Sobolev constant is close to that of the (suitably scaled) sphere, then the pushforward of the normalized volume measure by an almost extremal function is close (in Wasserstein distance) to what happens for the sphere (see Theorem \ref{thm:main result sobolev} for a precise statement). This form of stability was investigated for various functional inequalities \cite{CF25, Serres21,Serres23}, including spectral gap bounds under curvature-dimension conditions \cite{FGS24}. We start with a brief description of background material on Sobolev inequalities and curvature-dimension conditions, and we refer to \cite{Nobili25,BakryGentilLedoux14} for a more comprehensive survey. 

Any Riemannian manifold $(M,g)$ of dimension $N\ge 3$ and Ricci curvature bounded below by $N-1$ satisfies the following sharp Sobolev inequality 
  \begin{equation} \label{statement_sobolev}
    \|u\|_{L^{2^*}(\mu)}^2\le A_N\||\nabla u|\|_{L^2(\mu)}^2+\|u\|_{L^2(\mu)}^2, \quad \forall u \in W^{1,2}(M),
    \end{equation}
    where $2^*\coloneqq\frac{2N}{N-2}$ and $\mu$ is the renormalized volume measure and $A_N := \frac{2^*-2}{N}$ \cite{Ilias83,BL96,dupaigne2023sobolev,BakryGentilLedoux14}. Recall that any such manifold has diameter at most $\pi$  by Bonnet Myers theorem and thus has finite volume.  The constant $A_N$ is sharp, in that it is the sharp constant for the round sphere $\mathbb S^N$ of unit radius \cite{Aubin76a} and equality is achieved by the family of \textit{Aubin-Talenti bubbles} \cite{Aubin76f,Talenti}:
\begin{equation}\label{eq:AT}
    U_{a,b,x_0}(x) = (a-b \cos(\sfd_{\mathbb S^N}(\cdot, x_0)))^{1-N/2}; \hspace{3mm} a >0,\, b \in (-a,a) ,\, x_0\in \mathbb S^N,
\end{equation}
where $\sfd_{\mathbb S^N}$ denotes the geodesic distance. Note that $ U_{a,b,x_0}$ are either strictly positive or strictly negative radial functions.

For an arbitrary  $N$-dimensional manifold with Ricci curvature bounded below by $N-1$ on the other hand, inequality \eqref{statement_sobolev} might hold  with a strictly smaller constant.  As a matter of fact a natural question arises as  to whether the round sphere is the unique manifold, under the same curvature assumptions, for which $A_N$ can not be improved. This is reminiscent of Cheng's maximal diameter theorem \cite{Cheng1975}, which characterizes the sphere as the unique smooth manifold with Ricci curvature greater than $ N-1$ and diameter $\pi.$ It turns out that the same holds for the Sobolev inequality. More precisely, it was shown in \cite{NV22} that if  $A_N$  is the minimal constant such that \eqref{statement_sobolev} holds in $M$ then $M$ is isometric to a round sphere with unit radius.

With the full understanding of the optimality of \eqref{statement_sobolev} and the characterization of the equality case we can investigate its stability. 
Suppose that $(M,g)$ is $N$-dimensional, $N\ge 3$, with Ricci curvature bounded below by $N-1$ and that there exists a non-constant function $u$, called \textit{almost extremizer},   such that
$$\|u\|_{L^{2^*}(\mu)}^2\ge (1-\eps)A_N \||\nabla u|\|_{L^2(\mu)}^2+\|u\|_{L^2(\mu)}^2,$$
for some small $\eps> 0$. What can we deduce?  There are  mainly two directions that can be explored:
 \begin{enumerate}[label=\roman*)]
     \item \textit{Geometric stability}: Does the existence of an almost extremizer imply that the ambient manifold $M$ is close, in an appropriate sense, to the round sphere?
     \item \textit{Functional stability}: Do almost extremizers look like Aubin-Talenti bubbles?
 \end{enumerate}
 For the first question,  the answer is negative. In fact it was shown in \cite{NV22} that $M$ must instead be close--in the Gromov-Hausdorff topology--to a \emph{spherical suspension}, which is roughly speaking a non-smooth generalization of the round sphere. This stems from the fact that the compactification of the class of $N$-dimensional smooth manifolds with Ricci curvature bounded below by $N-1$ includes metric spaces with singularities. Within this broader class, the spaces that saturate the Sobolev inequality \eqref{statement_sobolev} are precisely the spherical suspensions of diameter $\pi$, among which the round sphere is the unique smooth representative. This is an instance of a more general phenomenon where extremal cases for functional inequalities,  under a Ricci curvature lower bound, are realized by singular metric-measure spaces. This perspective originates in the work of  Anderson \cite{Anderson90} about manifold with almost maximal diameter and was significantly developed in the subsequent works of Cheeger and Colding about Ricci limits \cite{CC96,CC97}. 

Concerning functional stability, an affirmative answer was given in \cite{NV24}, showing that almost extremizers are close to Aubin-Talenti bubbles. More precisely it is shown in \cite{NV24} that any almost extremizer is close in $W^{1,2}$-norm to a function $U_{a,b,x_0}$, given by the same expression in \eqref{eq:AT} where we replace $\sfd_{\mathbb S^N}(\cdot,x_0)$ with the geodesic distance from some point $x_0\in M$.  The underlying principle for this positive result is that, even if equality in \eqref{statement_sobolev} can be achieved in singular spherical suspensions, the extremal functions on these spaces are still bubbles  defined using the distance function from  one of the tips of the suspension.

The main drawback of these results however is that they are only \textit{qualitative}, in the sense that they lack  quantified estimates. This is not a technical point but rather an intrinsic limitation of strategy which crucially relyies on compactness arguments.

The aim of the present work is to obtain a first quantitative version of the stability for the sharp  Sobolev inequality \eqref{statement_sobolev}.  To do so we will use an approach which is, in some sense, an averaged version of the geometric  and functional versions  of the stability. 

Let $\mu_N$ be the normalized volume measure on $\mathbb S^N$ (i.e scaled so that $\mu_N(\mathbb S^N)=1$), $N\ge 3,$ and let $u$ be a non-constant positive extremal function such that $\|u\|_{L^{2^*}(\mu_N)}=1.$ Consider  the push forward of  $\mu_N$ via $u$:
\[
\nu\coloneqq u_\sharp\mu_N \in \mathcal P(\rr).
\]
All the possible outcome measures $\nu$ are easily computed and  coincide with the following one-parameter family of probability measures on the real line 
\begin{equation}\label{eq:muA}
      \left\{ \mu_A\coloneqq \frac{t^{-\frac{N}{N-2}}}{c_{A,N}}\left[(A^\frac{2}{N-2}-t^{\frac{2}{2-N}})(t^{\frac{2}{2-N}}-A^\frac{2}{2-N})  \right]^\frac{N-2}{2} \d t \restr{[A^{-1},A]}  
    \right\}_{A \in (1,\infty)},
\end{equation}
where $c_{A,N}>0$ is the  normalization constant enforcing unit total mass. 
In fact, the same holds in any spherical suspensions, with the same class $\{\mu_A\}_{A\in(1,\infty)}$ and where $\mu_N$ must be replaced by the ambient reference measure. 

Let now $M$ be a $N$-dimensional Riemannian manifold with Ricci curvature bounded below by $N-1$ and consider any almost extremal function $u$. Our goal will be to quantify how far $u_\sharp\mu $  is from the class  $\{\mu_A\}_{A\in(1,\infty)}$, where  $\mu$ is the renormalized volume measure in $M$ and the distance is measured  using the Wasserstein distance from optimal transport. 

In fact we will prove our main result for much more general spaces than Riemannian manifolds. As mentioned in the above discussion, the natural setting is the one of non-smooth metric spaces.  More precisely we consider  $\CD(K,N)$ and $\RCD(K,N)$ metric measure spaces having a weak notion of Ricci curvature bounded below by $K\in \rr$ and dimension bounded above by $N\in (-\infty,\infty]$ (here $N$ might be non-integer), see Section \ref{subsect_rcd}.  It is worth mentioning that the sharp Sobolev inequality \eqref{statement_sobolev}  holds also in this more general setting and with the same constant \cite{CM17b}. The reader interested only in the smooth case can replace   in the main statements below `$\CD(K,N)$' or `$\RCD(K,N)$'  by `smooth Riemannian manifold' with dimension bounded above by $N$ (when $N$ is non-negative) and Ricci curvature bounded below by $K$ or, more generally, by `weighted smooth Riemannian manifold'  with $N$-Bakry-\'Emery  Ricci tensor bounded below by $K$, see Section \ref{subsect_rcd}.

We are now ready to state our first main result.

\begin{theorem}[Stability of the sharp Sobolev inequality on $\CD(N-1,N)$ spaces]\label{thm:main result sobolev}
    Let $\Xdm$ be an essentially non-branching $\CD(N-1,N)$ space with $\mm(\X)=1$ and $N\in \mathbb{N}$ with $N\ge 3$. Set $2^*\coloneqq \frac{2N}{N-2}$.  Suppose that $u\in W^{1,2}(\X)$ satisfies 
    \begin{equation}\label{eq:reverse sobolev}
         \int u^{2^*}\d \mm =1, \quad \|u\|_{L^{2^*}(\mm)}^2\ge (1-\eps)A_N\||\nabla u|\|_{L^2(\mm)}^2+\|u\|_{L^2(\mm)}^2,
    \end{equation}
    for some $\eps>0$. Then it holds,  up to a multiplication of $u$ by $-1$,
    \begin{equation}\label{eq:eq:main result sobolev}
       \inf_{A> 1}   W_{2^*}(u_\sharp \mm, \mu_{A}) \leq C \sqrt{N\eps},
    \end{equation}
    where $C>0$ is an absolute constant and $\{\mu_A\}_{A>1}$ is the family of probability measures  in \eqref{eq:muA}.
\end{theorem}
The exponent $1/2$ in $\eps$ is sharp as shown in Example 2 of Section \ref{sec:opt}.
Essentially non-branching  is a technical condition which ensures that the ambient space does not posses too many pairs of branching geodesics (see Definition \ref{def:nonbr}). It is trivially satisfied on Riemannian manifolds and  holds in any $\RCD$ space \cite{RajalaSturm12}. This assumption is not too restrictive, as the  Sobolev inequality \eqref{statement_sobolev} itself is known to hold only for  essentially non-branching spaces.

The above result is not merely a weaker version of a functional stability. As a matter of fact  Theorem \ref{thm:main result sobolev} would not follow directly from an estimate on $\|u-U_{x_0,a,b}\|_{L^2(\mu)}$ (where $U_{x_0,a,b}$ is a bubble defined using the geodesic distance in $M$). Inequality \eqref{eq:eq:main result sobolev} should be instead interpreted as a joint averaged stability information about both the function $u$ and the reference measure $\mm$ together.  The main advantage of this approach is that enables us to obtain \eqref{eq:eq:main result sobolev} with explicit and simple dependence in $N$ and the sharp behavior $\sqrt \eps.$

In \cite[Theorem 1.4]{NV25polya}, under the same assumptions as Theorem \ref{thm:main result sobolev}, a quantitative statement of more geometric flavor was shown by proving that $(\pi-\diam(\X))\le C(N)\eps^{1/N}$ with a dimensional constant $C(N).$
A form of quantitative stability  in Riemannian manifolds for a  Sobolev inequality related to \eqref{statement_sobolev}  was also obtained in \cite{NP25}. The main difference is that the authors consider almost minimizers with respect to a fixed manifold $M$ with no curvature assumptions and estimate the distance from the non-explicit family of minimizers for $M$, rather than from Aubin-Talenti bubbles. Moreover the  constants appearing in the main result of \cite{NP25} depend implicitly on $M$, while here the dependence is  only on the dimension and the lower bound on the Ricci curvature in explicit form.

\begin{remark}
    From our argument we deduce  Theorem \ref{thm:main result sobolev}  also for \textit{non integer} $N>2$,   but without the explicit  dependence on $N$ on the constant in \eqref{eq:eq:main result sobolev}. The reason is that the $W^{1,2}$-strong stability result for the Sobolev inequality, with explicit dependence on $N$, in the weighted one-dimensional model space $I_N$  is known only for $N\in \mathbb N.$ 
    The difficulty is that the argument in the recent \cite{DEFFL22} relies strongly on the richer symmetries of $\rr^N$ for $N>1$ and is not immediately  adapted to the weighted one-dimensional interval. 
    \fr
\end{remark}

 Our strategy also extends to the logarithmic Sobolev inequality (LSI), which is the  analog of the Sobolev inequality as the upper bound on the dimension becomes infinite. Before giving the statement we recall the sharp inequality. In any $\RCD(1,\infty)$ space $\Xdm$ with $\mm(\X)=1$ it holds
\begin{equation}\label{eq:log sob}
   \frac12 \int u^2 \log (u^2)\d \mm\le \int |\nabla u|^2 \d \mm , \quad \text{for all $u\in W^{1,2}(\X)$,  $\|u\|_{L^2(\mm)}=1$,}
\end{equation}
see \cite{BakryGentilLedoux14} and references therein.
In the case of $\rr^N$ endowed with the standard Gaussian measure $\gamma$, this inequality was established by Gross \cite{Gross}, and its extension to positively curved manifolds goes back to Bakry and Emery \cite{BE85}. The inequality is sharp, and rigidity is characterized: if \eqref{eq:log sob} is attained for some function $u$ then the space \textit{splits off a Gaussian factor}, which means that
$\Xdm \simeq (\rr,\sfd_e,\gamma) \otimes (Y,\sfd_Y,\mm_Y), $
and, up to a sign, $u(t,y)=e^{-b^2+bt}$ for some $b \in \rr\setminus \{0\}$  (see \cite{OT20,Han21,ChengZhou17}). In $\rr^n$ this simply means that $u(x)=e^{b\cdot x-|b|^2}$ for some $b\in \rr^n$ (see also \cite{CARLEN1991194}). As for the Sobolev inequality, we can compute the push forward measure for these functions
\[
\nu=u_\sharp\mm=(e^{-b^2+bt})_\sharp(\gamma(t)\d t  \otimes \mm_Y(y))=:L_\sigma.
\]
This law $L_{\sigma}$ is known as the the Lognormal probability distribution of parameters $(-\sigma^2,\sigma^2)$, and its density is proportional to
$\frac1t {\rm exp}\left(-{\frac{(\log (t) + b^2)^2}{2b^2}}\right)$, $t>0$.
We can now state our stability result.

\begin{theorem}[Stability of the sharp LSI on $\RCD(1,\infty)$ spaces]\label{thm:main result log sobolev}
    Let $\Xdm$ be an $\RCD(1,\infty)$ space with $\mm(\X)=1.$ Suppose that $u\in W^{1,2}(\X)$ satisfies 
    \[
    \int u^2\d \mm =1; \quad  \int |\nabla u|^2 \d \mm -\frac12 \int u^2 \log (u^2) \d \mm \le \eps
    \]
   for some $\eps>0$. Then it holds,    up to a multiplication of $u$ by $-1$,
    \begin{equation}\label{eq:log sob stab}
      \inf_{{\sigma>0} } W_2(u_\sharp \mm, L_{\sigma} ) \le  C\sqrt{\eps},
    \end{equation}
    where $C>0$ is an absolute constant and $L_{\sigma}$ is the Lognormal distribution of parameters $(-\sigma^2,\sigma^2)$. 
\end{theorem}
 To our knowledge, such stability statement has previously only been established in the restricted setting of the Euclidean space endowed with a uniformly log-concave probability measure \cite{CF20}. For functional stability results of the LSI in Euclidean spaces see \cite{brigati2026logarithmic} and references there in. Also here the power $1/2$ for $\eps$ is sharp (see Example 1 in Section \ref{sec:opt}).

Our method  also applies to the spectral gap inequality for the Laplacian under various curvature-dimension assumptions. To keep the introduction brief, we refer to Section \ref{sec:gap} for the statement and further comments.

\medskip

\noindent \textbf{{Strategy}}

\medskip

The stability results that we obtain are  inspired by the ones  in \cite{FGS24,CF25} in the way we use the Wasserstein distance as a mean of measuring stability. However, the argument is substantially different.  These previous works  combined  quantitative differential estimates  with Stein's method for comparing probability distributions via approximate integration by parts formulas. Here we argue via symmetrization to reduce the problem to stability on the model space. Since the most difficult part of  \cite{FGS24,CF25} are the quantitative differential estimates, arguing by symmetrization simplifies the analysis, allowing us to reach \emph{non-quadratic} functional inequalities, thanks also to the results of \cite{DEFFL22} on quantitative functional stability for Sobolev inequalities on $\rr^N$ (see also \cite{Frank2024}). 

Recall that in the Euclidean space the Schwartz symmetrization of a (non-negative) function $u$ is a radial non-increasing function $u^*$ which is equimeasurable with $u$, meaning that $|\{u>t\}|=|\{u^*>t\}|$ for a.e.\ $t>0.$ The classical P\'olya-Szeg\H{o} inequality  says that the symmetrization process does not increase the Dirichlet energy. This fact has numerous applications in deducing sharp functional inequalities and study their rigidity properties (see e.g.\ \cite{BE91}).

It turns out that a rearrangement procedure can be performed even if the space is not symmetric as the Euclidean space. B\'erard and Meyer \cite{BM82} showed that a version of the P\'olya-Szeg\H{o} inequality holds on $n$-dimensional Riemannian manifolds with Ricci curvature bounded below by $n-1$, where the rearrangement $u^*$ is built as a radial non decreasing function on the round sphere $\mathbb S^n$. Later this result has been generalized to metric spaces with Ricci curvature bounded below, where $u^*$ needs to be defined on suitable one-dimensional model spaces \cite{MS20,NV25polya,Baernstein19}. See Section \ref{sec:rearrangement} for additional details. 

The main advantage for us to perform a symmetrization is that it behaves well under the process of taking the push forward measure.  The key observation is that the very equimeasurability of $u$ and $u^*$ ensures that, even if they live in different spaces endowed with different reference measures, the push forward of  said  measures via $u$ and $u^*$ respectively is the same. This will allow us  to effectively reduce the stability problem to quantitative  functional stability on a fixed one-dimensional space. 

It is worth mentioning that rearrangement techniques have previously been employed in the study of quantified stability for functional inequalities as a means of reducing problems to a simpler, more symmetric class of objects (see e.g.\ \cite{DEFFL22,FMP08,CFMP}).

	\section{Useful background}

    \subsection{Sobolev spaces on metric setting}
    We assume the reader to be familiar with the notion of Sobolev space $W^{1,2}(\X)$ in a metric measure space $\Xdm$ and refer to \cite{GP20,Bjorn-Bjorn11} for an introduction on the topic. To keep the note self-contained we recall the basic definitions. For any $f\in \LIP(\X) $  set 
    \[
    \lip f(x)\coloneqq \limsup_{y \to x} \frac{|f(x)-f(y)|}{\sfd(x,y)}, \quad \forall x\in \X,
    \]
    with the convention that $\lip f(x)\coloneqq 0$ if $x$ is isolated.

    \begin{definition}
        A function $f\in L^2(\mm)$ belongs to the space $W^{1,2}(\X)$
        if its Cheeger energy is finite:
        \[
        {\rm Ch}(f)\coloneqq \inf\left\{ \liminf_n \int \lip^2 f_n \d \mm  \ : \  f_n\overset{L^2(\mm)}{\longrightarrow} f,\, (f_n)\subset \LIP(\X)\cap L^2(\mm)  \right\}<\infty.
        \]
    \end{definition}
The space $W^{1,2}(\X)$ endowed with the norm  $\|f\|_{W^{1,2}(\X)}\coloneqq \sqrt{\|f\|_{L^2(\mm)}+{\rm Ch}(f)}$ is a Banach space.   Recall also that for any $f\in W^{1,2}(\X)$ there exists an appropriate function $|\nabla f|\in L^2(\mm)$, playing the role of the modulus of the gradient, such that ${{\rm Ch}(f)}=\int |\nabla f|^2\d \mm.$ For smooth functions on a Riemannian manifold $|\nabla f|$ is the usual norm of the gradient. 

In the one-dimensional case $\Xdm=(\overline I,\sfd_e, g\d t)$, where $I\subset \rr$ is an open interval and $g$ is a positive continuous function on $I$, it holds $f \in W^{1,2}(\X)$ if and only if $f \in W^{1,1}_{loc}(I)$ (in the usual sense) and $f,f'\in L^2(I;g \mathcal L^1)$, in which case $|\nabla f|=|f'|$  $\mathcal L^1$-a.e.\ in $I$  (see e.g.\ \cite[Lemma A.1]{NV25polya}).

	\subsection{Spaces with curvature bounded from below} \label{subsect_rcd}
	
    Let $(M^n,g)$ be an $n$-dimensional Riemannian manifold with (possibly empty) $C^2$ boundary $\partial M$ and let $V\in C^2(M)$ be positive. The modified $N$-Bakry-\'Emery Ricci tensor is given by
    \[
    {\rm Ric}_g^N\coloneqq {\rm Ric}_g+{\rm Hess}\, V  -\frac{\nabla V\otimes \nabla V}{N-n}, 
    \]
    where ${\rm Ric}_g$ is the Ricci tensor, $N\in (-\infty,\infty]$ and with the convention that $\frac{1}{\infty}=0$, $\frac1 0=\infty$ and $\infty\cdot 0=0.$

    The weighted Riemannian manifold $(M^n,g,e^{-V}\vol_g)$ is said to satisfy the \textit{Curvature Dimension condition} $\CD(K,N)$, $K\in \rr$ and $N\in (-\infty,\infty]$ if $M$ is geodesically convex and $  {\rm Ric}_g^N\ge Kg.$

    In the more general setting of metric measure spaces the notion of weak Ricci curvature lower bound was introduced independently by Sturm \cite{Sturm06II} and Lott and Villani \cite{Lott-Villani09} via optimal transport. 
 These  notions are compatible with the smooth setting in the sense that $(M^n,g,e^{-V}\vol_g)$ satisfies the synthetic $\CD(K,N)$ condition for $N>1$ and $K \in \rr\cup \{\infty\} $ if and only if ${\rm Ric}_g^N\ge  K g$ (see e.g.\ \cite{Sturm06II}).
    The synthetic version of the $\CD(K,N)$ condition for negative $N$ has also been introduced in the purely metric setting (see \cite{ohta16,MRS23}), however we will not consider it in this paper mainly because the sharp isoperimetric inequality has not been obtained  yet in this generality.

Finally recall that $\Xdm$ satisfies the \textit{ Riemannian Curvature Dimension condition} $\RCD(K,N)$ for $K\in \rr$ and $N\ge 1$ if satisfies both the $\CD(K,N)$ condition and it is \textit{infinitesimally Hilbertian}, that is  $W^{1,2}(\X)$ is a Hilbert space (see \cite{Gigli12}). For more background, variants and equivalence definitions of synthetic  Ricci curvature lower bounds we refer to the surveys \cite{AmbICM,Gigli23_working,sturmSurvey}.
    
To deal with sharp geometric and functional inequalities it is useful to restrict to the smaller class of essentially non-branching spaces introduced in \cite{RajalaSturm12}.
\begin{definition}\label{def:nonbr}
A metric measure space $\Xdm$ is called \textit{essentially non-branching} if  for any
$\mu_0,\mu_1\in  P_2(X)$, with $\mu_0$ absolutely continuous with respect to $\mm$, any element of ${\rm OptGeo}(\mu_0,\mu_1)$ is
concentrated on a set of non-branching geodesics.
\end{definition}
$\RCD$ spaces and in particular Riemannian manifolds are essentially non-branching.

	\subsection{One-dimensional model spaces}\label{sec:modelspaces}

	In this note we will work with three model spaces:
    \begin{itemize}
        \item 
    \textit{$\CD(N-1,N)$ model space} for $N>1$:  $I_N\coloneqq([0,\pi],\sfd_e,\mm_N)$, where $\sfd_e$ is the Euclidean distance and
		\[\mm_{N}:= \tfrac{1}{c_{N}}\sin^{N-1}\Leb 1\restr{[0,\pi]}, \]
		with $c_{N}:= \int_0^\pi\sin^{N-1}(t)(t)\, \d t.$
        \item 

    \textit{$\CD(N+1,-N)$ model space} for $N>1$:  $I_{-N}\coloneqq(\rr,\sfd_e,\mm_{-N})$, where 
		\[\mm_{-N}:= \tfrac{1}{c_{-N}}\cosh(t)^{-N-1}\Leb 1, \]
		with $c_{-N}:= \int_\rr\cosh^{-N-1}(t)(t)\, \d t.$
        \item  \textit{$\CD(1,\infty)$ model space}:  $I_{\infty}\coloneqq(\rr,\sfd_e,\gamma)$, where 
		\[\gamma:= \frac{1}{\sqrt {2\pi}}e^{\frac{-t^2}{2}}\Leb 1. \]
    \end{itemize}
Recall the sharp Sobolev inequality in $I_N$:
\begin{equation}\label{eq:sharp sobolev IN}
     \|u\|_{L^{2^*}(\mm_N)}^2\le \frac{2^*-2}{N}\|u'\|_{L^2(\mm_N)}^2+\|u\|_{L^2(\mm_N)}^2,   \quad \forall u \in W^{1,2}(I_N)
\end{equation}
 where $2^*=\frac{2N}{N-2}$ (see \cite{BakryGentilLedoux14}).
Equality is achieved, up to a change of sign, precisely by the following class of functions:
	\begin{equation}\label{eq:model bubble}
		v_{a,b}(t)\coloneqq \frac{1}{(a-b\cos(t))^\frac{N-2}{2}}, \quad t\in [0,\pi]; \quad a>0,\, b\in \rr,\, b\in(-a,a).
	\end{equation}
	Note that $v_{a,b}$ is a monotone function. A direct computation shows
\begin{equation}\label{eq:v 2*norm formula}
     \|v_{a,b}\|_{L^{2^*}(\mm_N)}=(a^2-b^2)^\frac{2-N}4.
\end{equation}
  In particular $\|v_{a,\sqrt{a^2-1}}\|_{L^{2^*}(\mm_N)}=1$ for all $a\ge 1.$
Restricting ourselves to the renormalized minimizers $v_{a,\pm \sqrt{a^2-1}}$, the push-forward 
	  $ (\pm v_{a,\pm \sqrt{a^2-1}})_\sharp\mm_N$ has the form $g_a(t)\mathcal L^1\restr{[A^{-1},A]},$
where $A\coloneqq (a+\sqrt{a^2-1})^\frac{N-2}2$ and 
    \begin{equation}\label{eq:foruma mu}
    \begin{aligned}
            g_a(t)= \frac{2t^{-\frac{N}{N-2}}}{c_N(\sqrt{a^2-1})^{N-1}(N-2)}\left[(A^\frac{2}{N-2}-t^{\frac{2}{2-N}})(t^{\frac{2}{2-N}}-A^{\frac{2}{2-N}})  \right]^\frac{N-2}{2}.
    \end{aligned}
    \end{equation}

	\section{Rearrangement in metric spaces}\label{sec:rearrangement}
For this part we borrow the notations from \cite[Section 3]{NV25polya}.
Let  $\Xdm$ be a metric measure space  and $u:\X\to \rr$ be a Borel function such that $\mm( \{u>t\})<\infty$ for any $t >\essinf u$. We define  $\mu:(\essinf u,+\infty)\to[0,\infty)$, the distribution function of $u$ as  $\mu(t)\coloneqq \mm(\{u> t\})$. For $u$ and $\mu$ as above, we consider the generalized inverse $u^{\#}:[0,\infty]\to [\essinf u,\esssup u]$ of $\mu$ defined by
\begin{equation*}
	u^{\#}(s)\coloneqq 
	\begin{cases*}
		{\rm ess}\sup u & \text{if $s=0$},\\
		\inf\left\lbrace t > \essinf u \ :\ \mu(t)<s \right\rbrace &\text{if $s>0$},\\
  \essinf u & \text{if $s=\infty$}.
	\end{cases*}
\end{equation*}
Note that $u^{\#}$ is non-increasing and left-continuous. Moreover, it holds
$$u^{\#}(s)={\rm ess}\inf u, \qquad \text{for all $s \ge  \mm(\X)$.}
$$

To define the decreasing rearrangement of $u$ on the real line, we  consider a weighted interval $(I,\omega)$, where $I\subset \R$ is a possibly unbounded open interval and $\omega=g \Leb 1\mres I$ for some continuous  function $g: I\to (0,\infty)$.  We also assume the following  basic conditions to be true:
\begin{align}
    &\mm(\X)\le \omega(I),   \label{eq:mass compatibility}  \\
&{\omega((-\infty,x])<\infty,\quad\forall x \in I,\qquad \lim_{x\to -\infty}\omega((-\infty,x])=0}. \label{eq:zero a sx}
\end{align}

We also define the usual \textit{cumulative distribution function} $F_\omega: \R \to [0,\infty] $ as 
\[
F_\omega(x) \coloneqq \omega((-\infty,x)),\qquad \forall \, x\in\R.
\]
For $u$ as above we define the \emph{decreasing rearrangement} $u^*:I \to [\essinf u,\esssup u]$ of $u$ with respect to  $\omega$ as
\[
I \ni x\mapsto u^*(x) \coloneqq u^\#(F_\omega(x)).
\]
Note that $u^*$ is finitely valued, monotone non-increasing and left-continuous as so it is $u^\#$. Note that we are not assuming that $u$ is non-negative.
By construction, $u^*$ is independent of the chosen representative of $u$ modulo $\mm$-a.e.\ equality. Therefore $u^*$ is in fact well defined for any $u \in L^1(\mm)$. It holds that $u$ and $u^*$ are equimeasurable, meaning that
\begin{equation}\label{eq:equi meas}
    \mm( \{u>t\}) = \omega( \{u^*>t\}), \quad \text{for all $t\in \rr.$}
\end{equation}
Moreover all continuous functions $G:\R \to \R$ it holds
    \begin{equation}
    \int_\X G(u) \,\d \mm= \int_{\rr} G(u^*)\, \d \omega,
    \label{eq:composition rearr G}
    \end{equation}
   in the sense that  one integral makes sense if and only if the other one does in which case equality holds (see \cite[Lemma 3.1]{NV25polya}). In particular $\| u\|_{L^p(\mm)} = \| u^*\|_{L^p(\omega) }$, for all $p \in [1,\infty]$. 

The cumulative distribution function $F_g\coloneqq F_\omega$ is strictly increasing in $I$ and admits an inverse $F_g^{-1}: (0,\omega(I))\to I$. We define the \textit{isoperimetric perimeter profile} function $\mathcal I^\flat_g(v): [0,\omega(I)]\to [0,\infty)$ as
\begin{equation}\label{eq:profile}
    \Ig(v)\coloneqq g(F_g^{-1}(v)), \text{ if $v>0$,} \quad \Ig(0)\coloneqq 0,\,\Ig(\omega(I))\coloneqq 0. 
\end{equation}

We report below in simplified form the P\'olya-Szeg\H{o} principle which was obtained in \cite{NV25polya}.
\begin{theorem}\label{thm:main Sobolev PZ metric}\label{thm:polya}
Let $\Xdm$ be a metric measure space and let $(I,\omega)$ be a weighted interval as above. Suppose that 
\begin{equation}\label{eq:isop polya pre}\tag{Iso}
   \Per(E)\ge \Ig(\mu(E)), \quad   \text{for all Borel sets $E\subset \X$}.
\end{equation}
Then for all $u\in W^{1,2}(\X)$ it holds $u^*\in W^{1,2}(I)$ and
\begin{equation}\label{eq:PZ abstract}
    \int_\X |\nabla u|^2\,\d\mm \ge \int_I |(u^*)'|^2\, \d \omega.
\end{equation}
\end{theorem}
The result above can be applied to all the model spaces introduced in Section \ref{sec:modelspaces} thanks to the following isoperimetric inequalities.

    \begin{theorem}[{Isoperimetric inequality for $\CD(N-1,N)$ spaces, \cite{CM17a}}]\label{thm:isop N}Let $\Xdm$ be an essentially non-branching $\CD(N-1,N)$ space, $N>1,$ with $\mm(\X)=1.$ Then \eqref{eq:isop polya pre} holds with $g(t)=c_N^{-1}\sin(t)^{N-1}$ and $I=(0,\pi).$
    \end{theorem}

     \begin{theorem}[{Isoperimetric inequality for $\RCD(1,\infty)$ spaces, \cite{AM16}}]\label{thm:isop inf}Let $\Xdm$ be an $\RCD(1,\infty)$ space with $\mm(\X)=1.$ Then  \eqref{eq:isop polya pre} holds with $g(t)=\frac{1}{\sqrt{2\pi}}e^{\frac{-t^2}{2}}$ and $I=\rr.$
    \end{theorem}

     \begin{theorem}[{Isoperimetric inequality for $\CD(N+1,-N)$ spaces, \cite{Mi17}}]\label{thm:isop -N} Let $(M,g, \mu)$ be a weighted smooth Riemannian manifold satisfying the  $\CD(N+1,-N)$ condition with $\mu(M)=1$ and $N>1$. Then  \eqref{eq:isop polya pre} holds with $g(t)= \tfrac{1}{c_{-N}}\cosh(t)^{-N-1}$ and $I=\rr.$
    \end{theorem}

	In the following result we note that the decreasing rearrangement preserves the pushforward of the underlying measure. This will play a key role in the proofs of the main theorems.  Even if it follows almost immediately from the definitions,  we were not able to find it explicitly stated in the literature. 
	\begin{lemma}\label{lem:le trick}
		Let $\Xdm$ be a metric measure space and let $(I,\omega)$ be a weighted interval as above. Then for any $u \in L^1(\mm)$, the pushforward of $\mm$ by the function $u$ and the pushforward of $\omega$ by its rearrangement $u^*$ coincide:
		\[
		u_\sharp\mm =u^*_\sharp\omega.
		\]
	\end{lemma}
	\begin{proof}
		From \eqref{eq:equi meas} we obtain that the two measures $ u_\sharp\mm$ and $u^*_\sharp\omega$ agree on the  family of open half lines $\{(t,\infty) \ : \ t \in \rr \}.$ Since they are both probability measures, they agree also on the closed half lines $\{-\infty, s] \ : \ s \in \rr \}.$ In particular they have the same cumulative distribution function and so they coincide by the $\pi$-$\lambda$ theorem.
	\end{proof}

\section{Quantitative  stability for Sobolev and log Sobolev inequalities}	

\subsection{One dimensional inequalities}
    The goal of this subsection is to show the following one-dimensional stability result, which we shall deduce from results of \cite{DEFFL22}.
	\begin{theorem}[Stability of Sobolev inequality on $\CD(N-1,N)$ model space]\label{thm:model Sobolev  quant stab}
		There is an universal explicit constant $\beta>0$ such that the following holds. For all $N\in \N$ with $N\ge 3$ and all monotone functions $u\in W^{1,2}(I_N)$,  up to a multiplication of $u$ by $-1$, it holds
		\begin{equation}\label{eq:model Sobolev  quant stab}
			\frac{\beta}{N} \inf_{a>0,\, b\in(-a,a)}  \frac{\|u-v_{a,b}\|_{L^2(\mm_N)}^2+\frac{2^*-2}{N}\|u'-v_{a,b}'\|_{L^2(\mm_N)}^2}{\|u'\|_{L^2(\mm_N)}^2}
			\le   \frac{2^*-2}{N} - \frac{\|u\|_{L^{2^*}(\mm_N)}^2-\|u\|_{L^{2}(\mm_N)}^2}{\|u'\|_{L^{2}(\mm_N)}^2}.
		\end{equation}
	\end{theorem}
The key feature of Theorem \ref{thm:model Sobolev  quant stab} is that the constant in front of the left hand side of \eqref{eq:model Sobolev  quant stab} is explicit. The same result with  non-explicit constant, for all $N>2$ not necessarily integer and without the monotonicity assumption,  can be easily deduced following the same the argument in \cite{BE91} (which relies on a compactness argument). 

   To prove Theorem \ref{thm:model log Sobolev  quant stab} we first need a result on the $N$-dimensional sphere. This says that stability for radially monotone functions can be measured restricting to  bubbles centered at the same point. By convention we will use the embedding $\mathbb S^N=\{x=(x_1,\dots,x_{N+1}) \ : \  |x|=1\}\subset \mathbb R^{N+1}$.
	\begin{proposition}\label{prop:quant stab sphere}
	    Let $U\in W^{1,2}(\ss^N)$ be  radial and monotone around the north pole, that is $U=\phi(x_{N+1})$ for some monotone function $\phi$. Then,  up to a multiplication of $u$ by $-1$, it holds
		\begin{equation}\label{eq:quant stab sphere}
			\begin{split}
				\frac{\beta}{N} \inf_{a>0,\, b\in(-a,a) } & \frac{\|U-G_{a,b}\|_{L^2(\mathbb S^N)}^2+\frac{2^*-2}{N}\|\nabla U-\nabla G_{a,b}\|_{L^2(\mathbb S^N)}^2}{\|\nabla U\|_{L^2(\mathbb S^N)}^2}
				\\
				& \quad \quad \quad \le   \frac{2^*-2}{N} - \frac{\|U\|_{L^{2^*}(\mathbb S^N)}^2-\|U\|_{L^{2}(\mathbb S^N)}^2}{\|\nabla U\|_{L^{2}(\mathbb S^N)}^2},
			\end{split}
		\end{equation}
        where $\beta>0$ is a universal explicit constant and $G_{a,b}(x)\coloneqq (a-bx_{N+1})^\frac{2-N}2.$
	\end{proposition}
	\begin{proof}
 We will pull back \eqref{eq:quant stab sphere} from a result on $\rr^N$ via the stereographic projection. In \cite[Theorem 1.1]{DEFFL22} it is shown that
         \begin{equation}\label{eq:stab sobolev rn}
            \|\nabla f\|_{L^{2}(\rr^N)}^{2}-|\ss^N|^\frac{2}{N}\frac{N}{2^*-2}\|f\|_{L^{2^*}(\rr^N)}^2\ge \frac \beta N \inf_{(x_0,c,d)\in \rr^n\times \rr \times \rr} \|\nabla f-\nabla g_{x_0,c,d}\|^2_{L^2(\rr^N)},
        \end{equation}
        where  $g_{x_0,c,d}(x)\coloneqq \frac{c}{(d+|x-x_0|^2)^\frac{N-2}2}$ and $\beta>0$ is a universal explicit constant. Suppose now that $f$ is non-negative and radially monotone non-increasing around the origin $O\in \rr^N.$ We claim that the infimum on the right hand side of \eqref{eq:stab sobolev rn} can be taken only among functions centered at the origin. This relies on the following identity:
        \begin{equation}\label{eq:pde trick}
        \begin{split}
             &\|\nabla f-\nabla g_{x_0,c,d}\|_{L^2(\rr^N)}^2=\|\nabla f\|_{L^2(\rr^N)}^2+\|\nabla g_{x_0,c,d}\|_{L^2(\rr^N)}^2+2\int_{\rr^n}f\Delta  g_{x_0,c,d}\\
             & \quad =\|\nabla f\|_{L^2(\rr^N)}^2+\|\nabla g_{x_0,c,d}\|_{L^2(\rr^N)}^2-\frac{|\ss^N|^\frac{2}{N}}{\| g_{x_0,c,d}\|_{L^{2^*}(\rr^N)}^{2^*-2}}\frac{N}{2^*-2}\int_{\rr^n}f g_{x_0,c,d}^{2^*-1},
        \end{split}
        \end{equation}
        (see e.g.\ the proof of \cite[Lemma 3.3]{DEFFL22}). The key observation is that by the Hardy-Littlewood inequality  we have that 
        \[
        \int_{\rr^n}f g_{x_0,c,d}^{2^*-1}\le \int_{\rr^n}f g_{O,c,d}^{2^*-1},
        \]
        which combined with \eqref{eq:pde trick}  and \eqref{eq:stab sobolev rn} shows 
         \begin{equation}\label{eq:stab sobolev rn radial}
            \|\nabla f\|_{L^{2}(\rr^N)}^{2}-|\ss^N|^\frac{2}{N}\frac{N}{2^*-2}\|f\|_{L^{2^*}(\rr^N)}^2\ge \frac \beta N \inf_{(c,d)\in \rr \times (0,\infty)} \|\nabla f-\nabla g_{O,c,d}\|_{L^2(\rr^N)},
        \end{equation}
        proving the claim. { Next, arguing exactly as in \cite[Section 3.2]{DEFFL22}, we can deduce that \eqref{eq:stab sobolev rn radial}, up to a change of the constant $\beta$, holds more generally when $f^+$ (the positive part of $f$) is non-negative, radially non-increasing and satisfies $\|f^+\|_{L^{2^*}(\rr^n)}\ge 1/2\|f\|_{L^{2^*}(\rr^n)} $.}

	    Denote by $S:\rr^{N} \to \ss^N\subset \rr^{N+1}$ the inverse of stereographic projection with pole $e_{N+1}$,  that is 
        \[
          \rr^N \ni x=(x_1,\dots,x_{N})\mapsto S(x)=\bigg(\frac{2x_1}{1+|x|^2},\dots,\frac{2x_N}{1+|x|^2},\frac{1-|x|^2}{1+|x|^2}\bigg).
        \]
    For all $F \in W^{1,2}(\ss^N)$ the function $\cS F(x)\coloneqq F(S(x))\left(\frac2{1+|x|^2}\right)^\frac{N-2}{2}$ belongs to $W^{1,2}(\rr^N)$ and the following identities hold
    \begin{equation}\label{eq:stero identities}
        \begin{split}
                &|\ss^N|^{\frac2N-1}\|\cS F\|_{L^{2^*}(\rr^N)}^2=\|F\|_{L^{2^*}(\ss^N)}^2\\
            &|\ss^N|^{-1}\|\nabla \cS F\|_{L^{2}(\rr^N)}^{2}=\|\nabla F\|_{L^{2}(\ss^N)}^2+\frac{N}{2^*-2}|\|F\|_{L^{2}(\ss^N)}^2
        \end{split}
    \end{equation}
        (see \cite[Section 2.1]{DEFFL22}).
        Moreover, for all $c\in \rr, d>0$ we can check that taking $b=(1-d)|c|^\frac{2}{2-N}, a=(1+d)|c|^\frac{2}{2-N}$ it holds
        \begin{equation}\label{eq:bubble=bubble}
            \cS ({\rm sign} (c) G_{a,b}) =g_{O,c,d}.
        \end{equation}
        Let now $U\in W^{1,2}(\ss^N)$ be as in the hypotheses, i.e.\ $U(x)=\phi(x_{N+1})$ with $\phi $ monotone. 
         Up to a multiplication of $U$ by $-1$, we can assume that $\|U^+\|_{L^{2^*}}\ge 1/2 \|U\|_{L^{2^*}}.$ Then, up to interchanging the north and south pole we can assume that $\phi$ is monotone non-decreasing.   Note that
        \[
        \cS U(x)=\phi\bigg(\frac{1-|x|^2}{1+|x|^2}\bigg)\left(\frac2{1+|x|^2}\right)^\frac{N-2}{2}.
        \]
        Then $(\cS U)^+=\cS (U^+)$ is a radially non-increasing function, as  $t\mapsto \frac{1-t^2}{1+t^2}$ is monotone non-increasing. {By the first in  \eqref{eq:stero identities} we have $\|(\cS U)^+\|_{L^{2^*}(\rr^n)}=\|\cS (U^+)\|_{L^{2^*}(\rr^n)}\ge 1/2\|\cS U\|_{L^{2^*}(\rr^n)} $. }   Hence  we can apply \eqref{eq:stab sobolev rn radial}  with $f=\cS U$, which combined with \eqref{eq:stero identities} and \eqref{eq:bubble=bubble}  shows \eqref{eq:quant stab sphere}.
	\end{proof}

We can now prove the main result of this section.
    \begin{proof}[Proof of Theorem \ref{thm:model Sobolev  quant stab}]
    We want to  push the result of Proposition \ref{prop:quant stab sphere} into the one-dimensional model space $I_N$.  To do so we consider the map $T:\mathbb S^N\to [0,\pi]$ given by $T(x_1,\dots,x_{N+1})=\arccos (x_{N+1})$.
	It is easy to check that for all $v \in W^{1,2}([0,\pi])$ we have $v\circ T\in W^{1,2}(\ss^N)$ and 
		\begin{equation}\label{eq:T isometry}
			\| v\circ T\|_{L^2(\mu)}=\|v\|_{L^2(\mm_N)}, \quad \quad 	\| \nabla v\circ T\|_{L^2(\mu)}=\|v'\|_{L^2(\mm_N)}, \quad \forall v \in \LIP([0,\pi]).
		\end{equation}
Moreover $v\circ T$ is radial around the north pole and it is monotone if and only if $v$ is monotone. Using the definitions we can also check the following
\begin{equation}\label{eq:G=v}
    G_{a,b}(x)=\frac{1}{(a-b x_{N+1})^\frac{N-2}{2}}=\frac{1}{(a-b \cos(T(x))^\frac{N-2}{2}}=v_{a,b}(T(x)) , \quad \forall x \in \mathbb S^N.
\end{equation}
	Combining \eqref{eq:G=v}, \eqref{eq:T isometry} and \eqref{eq:quant stab sphere} gives immediately \eqref{eq:model Sobolev  quant stab}. 
	\end{proof}

The following stability result for the LSI in the Gaussian model space is  already contained in \cite[Corollary 1.2]{DEFFL22}, and shall be used to prove Theorem \ref{thm:main result log sobolev}.
	\begin{theorem}[Stability of log-Sobolev inequality on $\CD(1,\infty)$ model]\label{thm:model log Sobolev  quant stab}
		There is an universal explicit constant $\beta>0$ such that the following holds. For all $u\in W^{1,2}(I_\infty)$  it holds
		\begin{equation}\label{eq:model log Sobolev  quant stab}
			 \beta\inf_{b\in \rr ,\, c\in \rr }  \int |u-ce^{bt}|^2 \d \gamma 
			\le  \int |u'|^2 \d \gamma -\frac12 \int u^2 \log \left (\frac{u^2}{\|u\|_{L^2(\gamma)}^2}\right)\d \gamma.
		\end{equation}
	\end{theorem}

\subsection{Proof of the main stability result}
Combining the estimates in the one-dimensional setting of the previous section with a rearrangement procedure we  now prove the main stability result for the Sobolev and log Sobolev inequalities.
    \begin{proof}[Proof of Theorem \ref{thm:main result sobolev}]
        Let $u \in W^{1,2}(\X)$ be as in the assumptions and consider its rearrangement $u^*$, as in Section \ref{sec:rearrangement}, in the model space $I_N$. Since the rearrangement preserves the $L^p$-norms, applying Theorem \ref{thm:polya} combined with the isoperimetric inequality in Theorem \ref{thm:isop N} yields
        \begin{equation}\label{eq:small deficit}
             1 -  \frac{N}{2^*-2}\frac{\|u^*\|_{L^{2^*}(\mm_N)}^2-\|u^*\|_{L^{2}(\mm_N)}^2}{\|(u^*)'\|_{L^{2}(\mm_N)}^2}\le   1 -  \frac{N}{2^*-2} \frac{\|u\|_{L^{2^*}(\mm)}^2-\|u\|_{L^{2}(\mm)}^2}{\||\nabla u|\|_{L^{2}(\mm)}^2}\le \eps.
        \end{equation}
        We claim that,  up to a multiplication of $u$ by $-1$,
        \begin{equation}\label{eq:L2 stab sobolev}
            \inf_{a>1,\, |b|=\sqrt{a^2-1}} \|u^*-v_{a,b}\|_{L^{2^*}(\mm_N)}^2< \frac{2N}{\beta}  \eps,
        \end{equation}
        where $\beta$ is an absolute positive constant. This would be sufficient to conclude. Indeed by Lemma \ref{lem:le trick}
        \begin{equation}\label{eq:coupling arg sob}
               W_{2^*}(u_\sharp\mm , (v_{a,b})_\sharp\mm_N  )=   W_{2^*}( (u^*)_\sharp\mm_N ,  (v_{a,b})_\sharp\mm_N)  \leq \|u^*-v_{a,b}\|_{L^{2^*}(\mm_N)},
        \end{equation}
        where we used $(u^*,v_{a,b})_\sharp\mm_N\in P(\rr^2)$ as admissible coupling.
        Inequality \eqref{eq:coupling arg sob} combined  with \eqref{eq:L2 stab sobolev} gives the required estimate.
        
       To show the above claim we note first that
        $$\lim_{a\to 1^+}\|u^*-v_{a,\sqrt{a^2-1}}\|_{L^{2^*}(\mm_N)}=\|u^*-1\|_{L^{2^*}(\mm_N)}\le\|u^*\|_{L^{2^*}(\mm_N)} +1 = \|u\|_{L^{2^*}(\mm)} +1\le  2.$$
       Hence \eqref{eq:L2 stab sobolev} trivially holds if $\eps \ge 1/2$, modulo changing the constant $\beta.$

        From now on we assume that $\eps <1/2.$
        Because $u^*$ is monotone non-increasing by construction, we can apply Theorem \ref{thm:model Sobolev  quant stab} combined with the  Sobolev inequality  \eqref{eq:sharp sobolev IN} to obtain
        \begin{equation}\label{eq:asdfg}
           \inf_{a>0,\, b\in(-a,a)}  \frac{\|u^*-v_{a,b}\|_{L^{2^*}(\mm_N)}^2}{\|(u^*)'\|_{L^2(\mm_N)}^2}\le  \frac{N}{\beta}   \frac{2^*-2}{N}\eps,
        \end{equation}
        modulo multiplication of $u^*$ by $-1.$
      Rearranging \eqref{eq:small deficit}     we can estimate $\|u'\|_{L^2(\mm_N)}^2$ as follows
         \begin{equation}\label{eq:derivative bound}
                 \|(u^*)'\|_{L^{2}(\mm_N)}^2\le \frac{N}{(2^*-2)(1- \eps)}\left(\|u^*\|_{L^{2^*}(\mm_N)}^2-\|u^*\|_{L^{2}(\mm_N)}^2\right)  \le   \frac{2N}{(2^*-2)},
         \end{equation}
         because $\|u^*\|_{L^{2^*}(\mm_N)}=1.$ Combining \eqref{eq:derivative bound} and \eqref{eq:asdfg} we obtain
         \begin{equation}\label{eq:baaaa}
            \inf_{a>0,\, b\in(-a,a)} \|u^*-v_{a,b}\|_{L^{2^*}(\mm_N)}^2< \frac{2N}{\beta}  \eps,
        \end{equation}
          modulo multiplication of $u^*$ by $-1.$ To get \eqref{eq:L2 stab sobolev}, we still have to show that we can take $a > 1$ and $|b|=\sqrt{a^2-1}$. Consider $a>0$ and $b\in(-a,a)$ so that \eqref{eq:baaaa} holds without the inf. Then $|\|v_{a,b}\|_{L^{2*(\mm)}}- 1 |\le \sqrt{\frac{2N}{\beta}  \eps}$. Taking $\tilde v_{a,b}\coloneqq \|v_{a,b}\|_{L^{2*(\mm)}}^{-1}v_{a,b}$, we have that $\tilde v_{a,b}=v_{\bar a, \pm \sqrt{\bar a^2-1}}$ for some $\bar a>1$ and with $\pm$ equals to ${\rm sign}(b)$ (recall \eqref{eq:v 2*norm formula}). Hence
        \[
        \|u^*-v_{\bar a, \pm\sqrt{\bar a^2-1}}\|_{L^{2^*}(\mm_N)}\le \|u^*-v_{a,b}\|_{L^{2^*}(\mm_N)} +  \sqrt{\frac{2N}{\beta}  \eps}\le  2\sqrt{ \frac{2N}{\beta}  \eps},
        \]
         modulo multiplication of $u^*$ by $-1.$ This shows \eqref{eq:L2 stab sobolev} up to a change of the constant $\beta.$
    \end{proof}

We now consider the log-Sobolev inequality.
\begin{proof}[Proof of Theorem \ref{thm:main result log sobolev}]
Let $u$ be as in the assumptions.
    Consider its rearrangement $u^*$ in the model space $I_\infty$, as in Section \ref{sec:rearrangement}. Then, since the rearrangement preserves the $L^p$-norms, by \eqref{eq:composition rearr G} applied with $G(t)=t\log(t)$ for $t\ge 0$ and $G(t)=0$ for $t\le 0,$  applying Theorem \ref{thm:polya} combined with the isoperimetric inequality in Theorem \ref{thm:isop inf} we get
    \begin{equation}
      \int |(u^*)'|^2 \d \gamma -\frac12 \int (u^*)^2 \log ((u^*)^2)\d \gamma\le    \int |\nabla u|^2 \d \mm -\frac12 \int u^2 \log (u^2)\d \mm\eqqcolon\eps.
    \end{equation}
    Therefore by Theorem \ref{thm:model log Sobolev  quant stab} and  modulo a  multiplication of $u^*$ by $-1,$
    \begin{equation}\label{eq:eqqq346}
        \inf_{b\in \rr\setminus\{0\} ,\, c>0 }  \int |u-ce^{bt}|^2 \d \gamma 
			\le\frac{\eps} {\beta}.
    \end{equation}
    The value $b=0$ in \eqref{eq:eqqq346} can be omitted because $ce^{bt}\to c$  in $L^2(\gamma)$ as $b\to 0$. We claim that in \eqref{eq:eqqq346} the infimum can be restricted to $c=e^{-b^2}$, up to a change of  $\beta.$  We first assume that $\frac{\eps} {\beta} \le \frac14 $. Note that by \eqref{eq:eqqq346} and the fact that $\|ce^{bt}\|_{L^2(\gamma)}=ce^{b^2}$, we deduce that
    \[
     1- \sqrt{\frac{\eps} {\beta}}\le e^{-b^2}e^{b^2}\le 1+ \sqrt{\frac{\eps} {\beta}}.
    \]
    Therefore
    \[
    \|ce^{bt}- e^{-b^2}e^{bt}\|_{L^2(\gamma)}=|c-e^{-b^2}|e^{b^2}=|ce^{b^2}- 1|\le \sqrt{\frac{\eps} {\beta}},
    \]
    which shows the claim.
    Hence using Lemma \ref{lem:le trick} as we did in \eqref{eq:coupling arg sob} we obtain
    $$\inf_{b\in \rr\setminus\{0\} }  W_2^2(u_\sharp\mm , (e^{-b^2}e^{bt})_\sharp\gamma)=W_2^2((u^*)_\sharp\mm_N , (e^{-b^2}e^{bt})_\sharp\gamma)\le  \frac{\eps} {\beta},$$
    up to increasing the constant $\beta.$
    The conclusion follows noting that $ (e^{-b^2}e^{bt})_\sharp\gamma$ is a Lognormal distribution with parameters $\mu=-b^2$ and $\sigma=|b|.$
\end{proof}

\subsection{Optimality}\label{sec:opt}

The following examples show that the order of magnitude $\sqrt{\eps}$ in  Theorem \ref{thm:main result sobolev} and Theorem \ref{thm:main result log sobolev} is optimal. 

\medskip

\noindent \textsc{Example 1}:
As in Section \ref{sec:rearrangement}, we denote by $F_\gamma$ the Gaussian cumulative distribution function. Let
$$
u_\eps(t) := \begin{cases}
    &\frac{\exp((1+\sqrt{\eps})t)}{\sqrt{L_\eps}}\text{ if } t \geq 0 \\
    &\frac{\exp((1-\sqrt{\eps})t)}{\sqrt{L_\eps}} \text{ if } t < 0,
\end{cases}
$$
where $L_\eps=e^{2(1+\sqrt{\eps})^2}F_\gamma(2+2\sqrt{\eps}) + e^{2(1-\sqrt{\eps})^2}F_\gamma(-2+2\sqrt{\eps})$ and $\eps>0.$

Since $\int_0^{+\infty}{\exp(cx)d\gamma} = e^{c^2/2}F_\gamma(c),$
we have
$\int{u_\eps^2d\gamma} = 1.$
Moreover an easy computation  yields
\[
\int_{-\infty}^{+\infty} |\nabla u_\eps|^2\d \gamma=1+\eps +2\sqrt{\eps} (I_+-I_-),
\]
where $I_+\coloneqq \int_0^\infty u_\eps^2\d\gamma$ and $I_-\coloneqq \int_{-\infty}^0 u_\eps^2\d\gamma$. Similarly
\begin{align*}
    \frac12\int u_\eps^2\log(u_\eps^2)\d \gamma&=\int  t u_\eps^2(t)\d \gamma(t)+ \sqrt{\eps} \left(\int_0^\infty tu_\eps^2\d\gamma -\int_{-\infty}^0 tu_\eps^2\d\gamma\right)-\log(\sqrt{L_\eps})\\
    &= 2 + 4\sqrt{\eps}(I_+-I_-)+2\eps + \frac{2\sqrt\eps}{L_\eps \sqrt{2\pi}}-\frac12\log({L_\eps}),
\end{align*}
where we integrated by parts in the second line. By a straightforward Taylor expansion, the $\sqrt{\eps}$-terms cancel out and we obtain
\[
\int |\nabla u_\eps|^2\d \gamma-  \frac12\int u_\eps^2\log(u_\eps^2)\d \gamma= O(\eps)
\]

Moreover, $u_\eps(t)$ can be written as a non-decreasing function of $e^{\sigma t-\sigma^2}$ as soon as $\sigma > 0$ (which can be assumed without loss of generality). Since the quadratic optimal transport map for real-valued distributions is the unique non-decreasing map sending one distribution onto the other, we have
\begin{align*}
W_2&((u_\eps)_\sharp \gamma, L_\sigma)^2 = \int{|u_\eps(t) - e^{\sigma t-\sigma^2}|^2d\gamma}= 2-2\int{u_\eps(t)e^{\sigma t-\sigma^2}d\gamma} \\
&= 2-2e^{-\sigma^2}L_\eps^{-1/2}\left(\int_{t > 0}{e^{(1+\sqrt{\eps} + \sigma)t}d\gamma(t)} + \int_{t<0}{e^{(1-\sqrt{\eps}+\sigma)t}d\gamma(t)}\right) = 2\left(1 - \frac{ g(\eps,\sigma) }{\sqrt{L_\eps}}\right),
\end{align*}
where
\begin{equation}\label{eq:g formula}
    g(\eps,\sigma) := e^{(1+\sqrt{\eps} +\sigma)^2/2-\sigma^2}F_\gamma(1+\sqrt{\eps} + \sigma) + e^{(1-\sqrt{\eps}+\sigma)^2/2-\sigma^2}F_\gamma(-1+\sqrt{\eps}-\sigma).
\end{equation}
Minimizing $W_2((u_\eps)_\sharp \gamma, L_\sigma)^2$ boils down to maximizing $ g(\eps,\sigma)$ with respect to $\sigma\in (0,\infty)$.
Since $F_\gamma'(t) = e^{-t^2/2}/\sqrt{2\pi}$, we have
$$\partial_2g(\eps,\sigma) = e^{-\sigma^2}\left((1+\sqrt{\eps}-\sigma)e^{(1+\sqrt{\eps}+\sigma)^2}F_\gamma(1+\sqrt{\eps} + \sigma) + (1-\sqrt{\eps}-\sigma)e^{(1-\sqrt{\eps}+\sigma)^2}F_\gamma(-1+\sqrt{\eps} - \sigma)\right).$$
Since $F_\gamma>0$, immediate considerations on the sign show that $\partial_2g(\sqrt{\eps},\sigma)> 0$ for $\sigma\in (0,1-\sqrt{\eps})$ and that  $\partial_2g(\sqrt{\eps},\sigma)$ can only vanish for $\sigma \in [1-\sqrt \eps, 1+\sqrt \eps]$. Moreover we can easily compute $\lim_{\sigma \to +\infty}g(\eps,\sigma)=0$. Therefore $g(\eps,\sigma)$ (with fixed $\eps$) posses a global maximum $\sigma_\eps\in [1-\sqrt\eps,1+\sqrt\eps]$. In other words $\sigma_\eps=1+\alpha_\eps$ for some $\alpha_\eps \in [-\sqrt \eps ,\sqrt \eps]$. Substituting this into \eqref{eq:g formula}, by Taylor expansion and the fact that $F_\gamma(2) + F_\gamma(-2) = 1$,
we end up with
$$g(\eps,\sigma_\eps) = e\left(1+   \sqrt \eps A + \frac{5}{2}\eps-\frac{\alpha_\eps^2}2+\sqrt \eps \alpha_\eps B+ o(\eps)\right ),$$
where $A=\left[2F_\gamma(2)-2F_\gamma(-2) + 2e^{-2}(2\pi)^{-1/2}\right]$ and $B=[F_\gamma(2)-F_\gamma(-2)]$ are  absolute constants.
Moreover, 
$$L_\eps = e^2\big([1+2\sqrt{\eps}A] + 10\eps + o(\eps)]\big).$$
Substituting the above expression in the formula for $W_2((u_\eps)_\sharp \gamma, L_{\sigma_\eps})^2$ and using once more a Taylor expansion, the $\sqrt \eps$-terms cancel out and we get 
$$W_2((u_\eps)_\sharp \gamma, L_{\sigma_\eps})^2=[5-A^2-2B\frac{\alpha_\eps}{\sqrt{\eps}}+\frac{\alpha_\eps^2}{\eps}]\eps + o(\eps)\ge[5-A^2-B^2]\eps + o(\eps), $$
where we used that $-2Bt+t^2\ge B^2$ for all $t\in \rr.$ Since $5-A^2-B^2\sim 0.02...$ we conclude.

\medskip

\noindent \textsc{Example 2}:
In the model space $I_N$, for each $\eps\ge 0$ we consider the function
$$
u_\eps(t) :=\begin{cases}
    & c_\eps\left(1+\frac{1}{2}\cos(t)+\sqrt{\eps}\cos(t)\right)^\frac{2-N}{2}\text{ if } t \in [0,\pi/2] \\
    & c_\eps\left(1+\frac{1}{2}\cos(t)-\sqrt{\eps}\cos(t)\right)^\frac{2-N}{2}\text{ if } t \in [\pi/2,\pi].
\end{cases}
$$
where $c_\eps$ is the constant so that $\|u_{\eps}\|_{L^{2^*}(\mm_N)}=1.$ For $\eps$ small enough $u_\eps$ is monotone non-decreasing. Note that $u_0=c_0(1+\cos(t)/2)^\frac{2-N}{2}$ coincides with the bubble $v_{1,-1/2}$ as defined in \eqref{eq:model bubble}. Straightforward Taylor expansions give
\[
\int u_\eps^2\d \mm_N=\int u_0^2\d \mm_N + \sqrt{\eps} \int u_0^2\left(\frac{2J}{2^*}+(2-N) g\right)\d \mm_N + O(\eps);
\]
\[
\int (u_\eps')^2\d \mm_N=\int (u_0')^2\d \mm_N + \sqrt{\eps} \int \frac{2J}{2^*}(u_0')^2+(2-N)u_0'(u_0g)'\d \mm_N + O(\eps),
\]
where $g(t)\coloneqq\frac{|\cos(t)|}{1+\cos(t)/2}\in \LIP([0,\pi])$ and $J\coloneqq N\int (u_0)^{2^*}g\d \mm_N.$ Since $u_0$ is a bubble, it satisfies equality in \eqref{eq:sharp sobolev IN} and the equation $-A_N\Delta u_0=u_0^{2^*-1}-u_0$. Using these facts and integrating by parts it can be easily checked that
\[
 \int u_0^2\left(\frac{2J}{2^*}+(2-N) g\right)\d \mm_N=A_N\int \frac{2J}{2^*}(u_0')^2+(2-N)u_0'(u_0g)'\d \mm_N.
\]
Hence we obtain that
\[
1=\|u_{\eps}\|_{L^{2^*}(\mm_N)}^2=A_N\|u_\eps'\|_{L^2(\mm_N)}^2+\|u_\eps\|_{L^2(\mm_N)}^2 + O(\eps),
\]
which combined with the fact that $\liminf_n \|u_\eps'\|_{L^2(\mm_N)}^2\ge \int (u_0')^2\d \mm_N >0$ gives that $u_\eps$ satisfies the reverse Sobolev inequality \eqref{eq:reverse sobolev} in the assumptions of Theorem \ref{thm:main result sobolev}, up to a multiplication of $\eps$ by a uniform constant factor depending only on $N$. On the other hand one can check that for any bubble $v_{a,-\sqrt{a^2-1}}$ with $a>1$ (defined as in Section \ref{sec:modelspaces}) it holds $\int |u_\eps-v_{a,-\sqrt{a^2-1}}|^2\d \mm_N\ge c\eps$ for a positive constant $c$ independent of $a$ and $\eps.$ This follows from the Taylor expansion $u_\eps=u_0 + \sqrt{\eps} (u_0\frac{J}{2^*}+\frac{2-N}{2}g)+ O(\eps)$ and explicit computations. Since any $\mu_A$ is of the form $( v_{a,-\sqrt{a^2-1}})_\sharp\mm_N$ for some $a>1$ and both $u_\eps$ and $v_{a,-\sqrt{a^2-1}}$ are monotone non-decreasing, we obtain that
\[
\inf_{A>1}  W_2((u_\eps)_\sharp \mm_N, \mu_{A})^2=\int |u_\eps-v_{a,-\sqrt{a^2-1}}|^2\d \mm_N\ge c\eps,
\]
which shows the sharpness of $\sqrt{\eps}$ in Theorem \ref{thm:main result sobolev}

\section{Stability of the spectral gap under various curvature dimension conditions}\label{sec:gap}

In this section we apply our method to obtain quantitative stability results for first the Laplacian eigenvalue. The goal is to prove Theorem \ref{thm: main result spectral gap} below, which is the analog of Theorems \ref{thm:main result sobolev} and \ref{thm:main result log sobolev} but for the spectral gap instead of Sobolev inequalities. It also includes the case of spaces with positive curvature and negative effective dimension. Let us start by introducing the relevant terminology and recall some known results.

In a metric measure space $\Xdm$ with finite measure the first eigenvalue of the Laplacian, also called \textit{spectral gap}, is given by
\begin{equation}\label{eq:spectral gap}
    \lambda_1(\X)\coloneqq \inf \left\{  \frac{\int |\nabla u|^2 \,\d \mm }{\int u^2\,\d \mm} \ : \  u \in W^{1,2}(\X)\setminus\{0\},\, \int u\, d \mm=0 \right\}.
\end{equation}
If $\X=M$ is a smooth closed Riemannian manifold then $\lambda_1(M)$ is precisely the first non-trivial eigenvalue of the Laplace-Beltrami operator in $M$ and a minimizer of \eqref{eq:spectral gap} satisfies $\Delta_M u=-\lambda_1(M) u.$

We collect below the known lower bounds for  $ \lambda_1(\X)$ under different curvature-dimension conditions, alongside their respective rigidity statements. 

\medskip

\noindent -For \textit{essentially non branching $\CD(N-1,N)$ spaces}, $N>1$, it holds
\begin{equation}\label{eq:lich}
    \lambda_1(\X)\ge N
\end{equation}
and equality holds if and only if $\X$ is a spherical suspension, in which case  $\lambda_1(\X)$ is achieved by $u=\cos (\sfd(x_0,\cdot))$, where $x_0$ is any of the tips of the suspension (see \cite{Ket15}).  The resulting push forward distribution $\nu_N\coloneqq u_\sharp\mm =Z_N^{-1}(1-t^2)^{N/2-1}\restr{[-1,1]}\d t$ is the symmetrized Beta distribution with parameters $(N/2, N/2)$.   For smooth Riemannian manifolds \eqref{eq:lich} is known as the Lichnerowicz inequality, and rigidity in that context was shown by Obata \cite{Obata62}.

\medskip

\noindent -For \textit{$\RCD(1,\infty)$ with $\mm(\X)=1$} it holds
\begin{equation}\label{eq:lichinf}
    \lambda_1(\X)\ge 1
\end{equation}
and equality holds if and only if $\Xdm \simeq (\rr,\sfd_e,\gamma) \otimes (Y,\sfd_Y,\mm_Y), $ in which case  $\lambda_1(\X)$ is achieved by $u(t,y)=t$ \cite{GKKS20}.  The resulting push forward distribution $u_\sharp\mm $ is  the standard Gaussian. In smooth weighted Riemannian manifolds the rigidity was shown in \cite{ChengZhou17}.

\medskip

\noindent -For \textit{weighted $\CD(N+1,-N)$ Riemannian manifolds $(M,g,\mu)$ with $\mu(M)=1$}, $N>1$,  it holds
\begin{equation}
    \lambda_1(\X)\ge N
\end{equation}
(see \cite{KM17}). Equality holds if and only if $(M,g,\mu)$ is isomorphic to a warped product $(\rr \times_{\cosh (t)} \Sigma,\cosh^{-N-1}(t)\d t \otimes \mm_{\Sigma})$, in which case  $\lambda_1(\X)$ is achieved by $u(t,y)=\sinh(t)$ \cite{Mai}.  The  push forward distribution $u_\sharp\mu$ becomes $\nu_{-N}\coloneqq Z_{-N}^{-1}(1+t^2)^{-N/2-1}\d t$.

We are now ready to state our stability result for the spectral gap.

\begin{theorem}\label{thm: main result spectral gap}
\begin{enumerate}[label=\roman*)]
  \item     Let $\Xdm$ be an essentially non-branching $\CD(N-1,N)$ space with $\mm(\X)=1$ and $N>1.$ If $u\in W^{1,2}(\X)$ satisfies 
  $$\int{u}\,\d \mm = 0; \quad \int u^2\d \mm = \frac{1}{N+1};\quad  \int{|\nabla u|^2\d \mm} \leq \frac{(1+\eps)N}{N+1}$$
  for some $\eps>0$, then 
  \begin{equation}\label{eq:gap stab positive N}
      W_2(u_\sharp\mm , \nu_N) \le  \sqrt{\frac{2\eps}{N+1}},
    \end{equation}
  where $\nu_N$ is the probability measure on $[-1,1]$ with density proportional to $(1-x^2)^{N/2-1}$.

  \item Let $\Xdm$ be an $\RCD(1,\infty)$ space with $\mm(\X)=1$.  If $u\in W^{1,2}(\X)$ satisfies
$$\int{u}\,\d \mm = 0; \quad \int u^2\d \mm = 1;\quad  \int{|\nabla u|^2\d \mm} \leq 1+\eps$$
   for some $\eps>0$, then 
  $$W_2(u_\sharp\mm , \gamma) \leq \sqrt{2\eps},$$
  where $\gamma$ is a standard Gaussian measure on $\R$. 
  \item Let $(M,g, \mu=e^{-V}\vol_g)$ be a geodesically convex weighted smooth Riemannian manifold with, $\mu(M)=1$, $V\in C^2(M)$ positive satisfying the  $\CD(N+1,-N)$ condition for $N>1$.   If $u\in W^{1,2}(M;\mu)$ satisfies
$$\int{u}\,\d \mu = 0; \quad \int u^2\d \mu = \frac{1}{N-1};\quad  \int{|\nabla u|^2\d \mu} \leq \frac{N(1+\eps)}{N-1}$$
   for some $\eps>0$, then 
  $$W_2(u_\sharp\mm , \nu_{-N}) \leq \sqrt{\frac {24\,\eps}{((N-1)\wedge  1)^3}},$$
  where $\nu_{-N}$ is the probability measure on $\rr$ with density proportional to $(1+x^2)^{-N/2-1}$. 
\end{enumerate}
\end{theorem}
Let us compare the above result with the existing literature.  A quantitative stability results for the spectral gap in essentially non-branching  $\CD(N-1,N)$ spaces for $N>1$  was obtained in \cite{CMS23} by showing $\|u-\sqrt{N+1}\cos(\sfd(x_0,\cdot))\|_{L^2(\mm)} \le C(N)\eps^\frac{1}{6N+4}.$  Additionally, a  very similar result to Theorem \ref{thm: main result spectral gap} was obtained in \cite{FGS24}. The comparison with our  result in this case is somewhat subtle.  For $\RCD(N-1,N)$ spaces with $N>1$ i) they show  that if $u$ \textit{is an eigenfunction} with eigenvalue $\lambda \leq N + \eps$, then
\begin{equation}\label{eq:w1 estimate}
    W_1(u_\sharp\mm , \nu_N) \leq C(N)\eps.
\end{equation}
 At first glance this is a better rate than ours, since Theorem \ref{thm: main result spectral gap} combined with the H\"older inequality  would imply \eqref{eq:w1 estimate} only with $\sqrt{\eps}$. The same rate was obtained for $\CD(N+1,-N)$ manifolds and $\RCD(1,\infty)$ spaces (up to a $\log(1/\eps)$ factor).  However, the assumption that $u$ is exactly an eigenfunction in \cite{FGS24} is much stronger than simply assuming the Rayleigh quotient to be close to the spectral gap, as we do here. As  a matter of fact, it turns out that  the rate $\sqrt{\eps}$ in Theorem \ref{thm: main result spectral gap} is sharp for all cases i),ii) and iii). For case i) we can take the round sphere $\mathbb S^N$ and $u$ the form $\sqrt{1-\eps}u_1 + \sqrt{\eps}u_2$ where $u_1$ is the first non-trivial eigenfunction and $u_2$ any function orthogonal to $u_1.$ Note that  this choice is ruled out by the assumptions of \cite{FGS24}.  Similar examples can be built for the infinite and negative dimensional cases ii) and iii). Finally we note  that the results in \cite{FGS24} are for $\RCD$ spaces, while ours apply also to the broader class of essentially non-branching $\CD$ spaces.

The proof of Theorem \ref{thm: main result spectral gap}  follows the same strategy as Theorem \ref{thm:main result sobolev}. We will  first obtain quantitative stability in the model spaces and then apply the rearrangement tools in Section \ref{sec:rearrangement}. Functional stability in the model space amounts to the spectral gap being an isolated eigenvalue of the Laplacian of the model space.

\begin{theorem}[Spectral-gap stability on the $I_N$ model space]\label{thm:1-d stability gap N}
    Fix $N>1$. Then for all $f \in W^{1,2}(I_N)$ with $\|f\|_{L^2(\mm_N)}^2=\frac{1}{N+1}$ and $\int f\d \mm_N=0$  it holds
    \begin{equation}\label{eq:1-d stability gap N}
       \min \left\{ \|f-\cos(t)\|_{L^2(\mm_N)}^2,\|f+\cos(t)\|_{L^2(\mm_N)}^2\right\}\le \frac{2N}{(N+2)(N+1)}\left(\frac1N\frac{\int (f')^2\d \mm_N}{\int f^2 \d \mm_N}-1\right).
    \end{equation}
\end{theorem}
\begin{proof}
The argument is based on spectral analysis. The Laplacian operator $-\Delta_{\mm_N} = \frac{\d^2}{\d t^2}+(N-1)\frac{\cos(t)}{\sin(t)} \frac{\d}{\d t}$ is known to have discrete  spectrum $0=\lambda_0<\lambda_1<\lambda_2<\dots \le \lambda_k\to +\infty$ with eigenvalues given by the formula $\lambda_k = k(k+N-1)$ (see e.g.\ \cite[Section 4.2]{orthpol} after the change of variable $x=\cos(t)$). 
By scaling we can  assume $\|f\|_{L^2(\mm_N)}=1$. 
Expanding $f = \sum_{i=1}^\infty a_i \phi_i$ with an orthonormal basis of eigenfunctions we obtain
\begin{align*}
    \int |f'|^2 \d \mm_N &= \sum_{i=1}^\infty a_i^2 \lambda_i \ge a_1^2 N + \lambda_2 (1-a_1^2) = 2N + 2 - (N+2)a_1^2.
\end{align*}
Rearranging terms
\[
    \frac{1}{N}\int |f'|^2 \d \mm_N - 1 \ge \frac{N+2}{N}(1-a_1^2) \ge \frac{N+2}{N} \min\{1-a_1, 1+a_1\}.
\]
The conclusion follows from the identity $\frac{1}{2}\|f \pm \phi_1\|_{L^2(\mm_N)}^2 = 1 \pm a_1$, with $\phi_1=\sqrt{N+1}\cos(t).$
\end{proof}
Note that \eqref{eq:1-d stability gap N} is sharp, since it becomes an equality choosing $f=u_2$ the second (non trivial) renormalized eigenfunction of the Laplacian in $I_N$ relative to the eigenvalue $\lambda_2=2(N+1).$  A similar statement  appeared before in \cite[Prop.\ 3.6]{CMS23} for all $\CD(N-1,N)$  weighted intervals but with  the deficit on right hand side raised to a worst exponent (here we have one while they obtain $\min (1,2/N)$). Additionally as in \cite{CMS23}, the stability in  Theorem \ref{thm:1-d stability gap N} can be obtained for full $W^{1,2}$-norm  with the same argument and up to change the multiplying constant.

\begin{theorem}[Spectral-gap stability on the $I_\infty$ model space]\label{thm:1-d stability gap inf}
    For all $f \in W^{1,2}(I_{\infty})$ with $\|f\|_{L^2(\gamma)}=1$ and $\int f\d \gamma=0$ it holds
    \begin{equation}\label{eq:1-d stability gap inf}
        \min \left\{\|f(t)-t\|_{L^2(\gamma)}^2,\|f(t)+t\|_{L^2(\gamma)}^2\right\}\le2 \left({\int |f'|^2\d \gamma}-1\right).
    \end{equation}
\end{theorem}
\begin{proof}
    The argument is the same as in Theorem \ref{thm:1-d stability gap N}, using that $\lambda_1=1$ and $\lambda_2=2$. Indeed the Laplacian operator $\Delta_\infty$ in $I_\infty$ is  $\Delta_\infty (u)=u''-tu$ and thus its eigenfunctions are the standard Hermite polynomials with eigenvalues $\lambda_n=n.$
\end{proof}
\begin{theorem}[Spectral-gap stability on the $I_{-N}$ model space]\label{thm:1-d stability gap -N}
    For all $f \in W^{1,2}(I_{-N})$ with $\|f\|_{L^2(\mm_{-N})}^2=\frac{1}{N-1}$ and $\int f\d \mm_{-N}=0$ it holds
    \begin{equation}\label{eq:1-d stability gap -N}
        \min \left\{\|f(t)-\sinh(t)\|_{L^2(\mm_{-N})}^2,\|f(t)+\sinh(t)\|_{L^2(\mm_{-N})}^2\right\}\le \frac {16}{((N-1)\wedge  1)^3} \left(\frac{{\int |f'|^2\d \mm_{-N}}}{N}-1\right).
    \end{equation}
\end{theorem}
\begin{proof}
The argument follows again Theorem \ref{thm:1-d stability gap N}. By density, it suffices to prove the statement for $f \in C^\infty_c(\rr)$ with $\|f\|_{L^2(\mm_{-N})}=1$. Set $g \coloneqq (c_{-N})^{-1}\cosh(t)^{-(N+1)}$ and $v =\sqrt g f $. A direct computation yields:
\begin{equation}\label{eq:laplacian change of variable}
-\Delta_{\mm_{-N}} f = \left( -v'' + V(t) v \right) g^{-1}, \quad \text{where } V(t) = \frac{(N+1)^2}{4} - \frac{(N+1)(N+3)}{4\cosh^2(t)}.
\end{equation}
The operator $L = -\frac{\d^2}{\d t^2} + V(t)$ is a shifted Schr\"odinger operator. Its spectral properties are standard (see \cite[Pag.\ 94]{flugge}, \cite[Theorem XIII.15]{reedsimon}): $L$ is self-adjoint on $H^2(\rr)$ with bottom eigenvalue $\lambda_1(L)=N$ and essential spectrum $\sigma_{ess}(L) = [\frac{(N+1)^2}{4}, \infty)$.   Observe that
\[
\delta \coloneqq d(N, \sigma(L) \setminus \{N\}) = \begin{cases} \frac{(N-1)^2}{4}, & 1<N<3, \\ |N-2|, & N\ge 3. \end{cases}
\]
From \eqref{eq:laplacian change of variable} we have $E_{\lambda_1}={\rm span}\{\phi_1\},$ with $\phi_1=\sqrt{N-1}\sqrt{g}\sinh(t)$, $\|\phi_1\|_{L^2(\rr)}=1.$ Decompose $v = a_1 \phi_1 + v^\perp$. Applying the spectral theorem \cite[Theorem 9.38]{hall} for $L$
\begin{align*}
 \int_{\rr} |f'|^2 \d\mm_{-N} &=- \int f(\Delta_{\mm_{-N}}f )\,g(t) \d t= \la L v, v \ra_{L^2(\rr)} = \int_{\sigma(L)} \lambda \, \d\mu_v(\lambda) \\
&\ge N \|a_1\phi_1\|_{L^2(\rr)}^2 + (N+\delta)\|v^\perp\|_{L^2(\rr)}^2 =N + \delta - \delta a_1^2.
\end{align*}
We conclude as in Theorem \ref{thm:1-d stability gap N}, using that $1 \pm a_1 = \frac{1}{2} \|f \pm \sqrt{N-1}\sinh(t)\|^2_{L^2(\mm_{-N})}$. Note also that by direct computation $\frac{2N}{\delta(N-1)}\le \frac {16}{((N-1)\wedge  1)^3}$.
\end{proof}

\begin{proof}[Proof of Theorem \ref{thm: main result spectral gap}]
We consider only case i), the others are completely analogous.
Let $u$ be as in the assumptions.
    Consider its rearrangement $u^*$ in the model space $I_N$, as in Section \ref{sec:rearrangement}. Then $\|u^*\|^2_{L^2(\mm_N)}=(N+1)^{-1}$. Applying Theorem \ref{thm:polya} combined with the isoperimetric inequality in Theorem \ref{thm:isop N} yields
    \begin{equation}
      \frac1N\frac{\int |(u^*)'|^2 \d \mm_N}{\int (u^*)^2\d \mm_N}-1 \le   \frac1N\frac{\int |\nabla u|^2 \d \mm}{\int u^2\d \mm} -1\le \eps.
    \end{equation}
    Therefore by Theorem \ref{thm:1-d stability gap N} 
    \begin{equation}
          \min \left\{ \|u-\cos(t)\|_{L^2(\mm_N)}^2,\|u+\cos(t)\|_{L^2(\mm_N)}^2\right\}\le \frac{2N}{(N+2)(N+1)}\eps.
    \end{equation}
    Next we note that $(\cos(t))_\sharp\mm_N=(-\cos(t))_\sharp\mm_N =\nu_N$.
    Hence as before, using  Lemma \ref{lem:le trick}  we obtain
    $$W_2^2(u_\sharp\mm , \nu_N)\le \frac{2N}{(N+2)(N+1)}\eps.$$ 
\end{proof}

	\textbf{{Notations and  objects used in the note}}:
     A \textit{metric measure space} $\Xdm$ is a triple where $(\X,\sfd)$ is a complete and separable metric space and $\mm$ is non-negative Borel probability measure. Given a Borel map $T:(\X,\sfd_\X)\to (Y,\sfd_Y)$ between two metric spaces and a non-negative Borel measure $\mu$ in $\X$, we denote by $T_\sharp\mu$ the push-forward measure defined by $T_\sharp\mu(E)\coloneqq \mu(T^{-1}(E))$.
     $P_p(\X)$, $p\in [1,\infty),$ denotes the space of probability measures with finite $p$-th moment in $(\X,\sfd).$ $W_p: P_p(\X)\times P_p(\X) \to [0,\infty)$ denotes the $p$-Wasserstein distance. $\LIP(\X)$ denotes the space of Lipschitz functions on $\X.$ $\sfd_e$ denotes the Euclidean distance in $\rr$ and $\mathcal L^1$ the one-dimensional Lebesgue measure.

\medskip

	\textbf{{Acknowledgments}}  The first author was supported by the Agence Nationale de la Recherche (ANR) Grant ANR-23-CE40-0003 (Project CONVIVIALITY). The second author is funded by the European Union (ERC, ConFine,
101078057). We wish to thank Ivan Gentil and Jean Dolbeault for useful discussions about stability for Sobolev inequalities and Francesco Nobili for valuable comments on a preliminary draft of the manuscript. 

\medskip

\textbf{{Competing interests}} 

The authors have no competing interests to declare that are relevant to the content of this article.

\end{document}